\documentclass[12pt,reqno]{amsart}

 \usepackage{mathrsfs}

\usepackage{color}
\usepackage{amsmath, amsthm, amssymb}
\usepackage{amsfonts}
\usepackage[ansinew]{inputenc}
\usepackage{graphicx}
\usepackage{caption}
\usepackage{subcaption}

\usepackage[english]{babel}
\usepackage{hyperref}

\usepackage{tikz}
\usepackage{rotating}

\usepackage{cite}
\usepackage{amscd}
\usepackage{color}
\usepackage{bm}
\usepackage{enumerate}

\usepackage{verbatim}
\usepackage{hyperref}
\usepackage{amstext}
\usepackage{latexsym}
\theoremstyle{plain}
\newtheorem{theorem}{Theorem}[section]

\newtheorem{lem}[theorem]{Lemma}

\numberwithin{theorem}{section} \numberwithin{equation}{section}

\newtheorem {thm} {Theorem}

\newtheorem {prop} [thm] {Proposition}
\newtheorem {lem{lem}}[thm]{Lemma}
\newtheorem {cor} [thm] {Corollary}
\newtheorem {defn} {Definition}
\newtheorem {rmk} {Remark}

\newcommand{\average}{{\mathchoice {\kern1ex\vcenter{\hrule height.4pt
width 6pt depth0pt} \kern-9.7pt} {\kern1ex\vcenter{\hrule height.4pt
width 4.3pt depth0pt} \kern-7pt} {} {} }}

\def\R{\mathbb{R}}

\renewcommand{\d}{\delta }

\newcommand{\D }{\Delta }

\newcommand{\n }{\nabla }

\renewcommand{\phi}{\varphi}

\renewcommand{\t }{\tau }

\renewcommand{\O }{\Omega }

\newcommand{\be}{\begin{equation}}
\newcommand{\ee}{\end{equation}}

\newcommand{\de}{\partial}

\newcommand{\ra}{{\rangle}}
\newcommand{\la}{{\langle}}

\newcommand{\N}{\mathbb{N}}

\renewcommand{\H}{{\mathcal H}}

\newcommand{\cF}{{\mathcal F}}

\newcommand{\cL}{{\mathcal L}}

\newcommand{\cO}{{\mathcal O}}

\newcommand{\Ss}{\textbf{S}}

\newcommand{\B}{{\bf B}}

\renewcommand{\epsilon}{\varepsilon}

\begin{document}

\title {Foliation of an asymptotically flat end by critical capacitors}
%\author[M. Christian] {Murray Christian}
%\address {???}
\author[M. M. Fall] {Mouhammed Moustapha Fall}
\address{African Institute for Mathematical Sciences in Senegal}
\email{\small{mouhamed.m.fall@aims-senegal.org} }

\author[I. A. Minlend]{Ignace Aristide Minlend} 
\address{Faculty of Economics and Applied Management, University of Douala}
\email{\small{ignace.a.minlend@aims-senegal.org} }

\author[J. Ratzkin] {Jesse Ratzkin} 
\address{Institut f\"ur Mathematik, Universit\"at W\"urzburg}
\email{\small{jesse.ratzkin@mathematik.uni-wuerzburg.de} }
%{Murray Christian, Mouhammed Moustapha Fall, \\ Ignace Aristide Minlend,
%and Jesse Ratzkin}
\maketitle

\begin {abstract}\noindent We construct a foliation of an asymptotically flat 
end of a Riemannian manifold by hypersurfaces which are critical points of 
a natural functional arising in potential theory. These hypersurfaces are 
perturbations of large coordinate spheres, and they admit solutions of 
a certain over-determined boundary value problem involving the Laplace-Beltrami 
operator. In a key step we must invert the Dirichlet-to-Neumann operator, 
highlighting the nonlocal nature of our problem. \\
\end {abstract}


\section {Introduction}

Riemannian manifolds with asymptotically flat ends play an important 
role in general relativity and cosmology, and so their general 
properties are of great interest. In particular, it is often useful to 
foliate an asymptotically flat end with special surfaces. Huisken and 
Yau \cite {HY} famously proved one can foliate a three-dimensional, 
asymptotically flat end with constant mean curvature spheres. 
Furthermore they prove these spheres share a common center, 
which one can take as the physical center of mass of the system. 
Previously, R. Ye \cite{RY} had shown one can foliate an asymptotically 
flat end in any dimension $n \geq 3$ provided the mass at infinity is 
nonzero. 

Subsequently, others have found special foliations by constant expansion 
surfaces \cite{M}, by Willmore 
surfaces \cite{LMS}, and by isoperimetric surfaces \cite{EM}. Here 
we investigate surfaces which are critical points of the Newton 
capacity. Recall that, if $K \subset \R^n$, with
$n\geq 3$, is a compact set, one can define its Newton capacity
as
\begin {equation} \label{cap-defn} 
\operatorname{Cap}(K) = \frac{1}{n(n-2)\omega_n} \inf \left \{ 
\int_{\R^n} |\nabla u|^2 dx : u \in H^1(\R^n), \left. u \right |_K \equiv 1
\right \}, \end {equation} 
where $\omega_n$ is the Euclidean volume of an $n$-dimensional unit 
ball and $H^1(\R^n)$ is the Sobolev space of functions with one weak 
derivative in $L^2(\R^n)$. Standard results in potential theory imply this infimum is 
realized by the equilibrium potential function $U_K$, which solves 
the boundary value problem 
\begin {equation} \label{cap-pde}
\Delta_0 U_K = 0  \textrm{ in } \Omega = \R^n \backslash K, \qquad 
\left. U_K \right |_{\partial K} = 1, \qquad \lim_{|y| \rightarrow \infty} 
U_K (y) = 0 \end {equation} 
where $\Delta_0$ is the usual, flat Laplacian. Moreover, the solution to 
\eqref{cap-pde} is unique among all functions which satisfy an 
appropriate decay condition. 

It is straight-forward 
to generalize both \eqref{cap-defn} and \eqref{cap-pde} to the 
setting of a compact set $K$ in a
complete, noncompact Riemannian manifold $(M,g)$ with an 
asymptotically flat end. As discussed in \cite{SY2}, Newton
 capacity plays a role in the study of scalar curvature and conformal 
 geometry.
 
  The functional $\operatorname{Cap}$ is not scale-invariant in Euclidean 
 space, one should not expect it to have critical points as a domain 
 functional. Thus it is natural to seek critical, and even extremal, 
 domains either subject to a constraint or of a modified functional 
 which is scale invariant. One can normalize $\operatorname{Cap}$ 
 using the volume of $K$ or the surface area of $\partial K$; both 
 choices are natural and have roots in physics and potential theory 
 \cite{GNR}. Below we will seek critical sets of a volume-normalized 
 functional, which leads us to the over-determined boundary value problem 
 \begin {equation} \label {eq:capacitor0}
\Delta_0 u = 0 \textrm{ on }\R^n \backslash K, \qquad \left.
 u \right |_{\partial K} = 1 , \qquad \lim_{|x| \rightarrow \infty}
 u(x) = 0, \qquad \left . \frac{\partial u}{\partial \eta} \right |_{\partial K}
 = \Lambda , \end {equation}
 where $\Lambda$ is a constant. See Section \ref{sec:euler-lagrange} 
 for a derivation of \eqref{eq:capacitor0} as the Euler-Lagrange 
 equation of our normalized domain functional. This computation 
 is standard, but we include it in Appendix \ref{sec:euler-lagrange} 
 for the reader's convenience. 
 
 Of course, one can also choose to normalize using the surface area
 of $\partial K$, which leads one to a slightly different over-determined boundary 
 value problem, namely
 $$
\Delta_0 u = 0 \textrm{ on }\R^n \backslash K, \qquad \left.
 u \right |_{\partial K} = 1 , \qquad \lim_{|x| \rightarrow \infty}
 u(x) = 0, \qquad \left . \frac{\partial u}{\partial \eta} \right |_{\partial K}
 = \Lambda H,$$
 where $H$ is the mean curvature of $\partial K$. We derive 
 this Euler-Lagrange equation as well, even though we do not 
 require it. 

 A classical theorem of Serrin \cite{Ser} implies that the only critical capacitors in
 Euclidean space are round sphere, but one expects the situation to be more
 complicated in a general Riemannian manifold. 

We introduce some notation so that we can state our main 
theorem. Our setting is that of a Riemannian manifold $(M,g)$  
of dimension $n\geq 3$
with one asymptotically flat end. In other words, there exists a
compact set $K \subset M$ and a diffeomorphism
\begin {equation} \label{end-parameterization1}
\Phi: \R^n \backslash \B \rightarrow M \backslash K, \end {equation}
such that in these coordinates
\begin {equation} \label{metric-expansion1}
g_{ij}(y) = (1+ \sigma |y|^{1-n}) \delta_{ij} + h_{ij}(y), \quad
|h_{ij}(y)| = \mathcal{O}(|y|^{-n}), \quad \partial^k h_{ij}(y) =
\mathcal{O}(|y|^{-k-n}).
\end {equation}
Here $\B$ is the unit ball in $\R^n$ centered at the origin and
$\partial^k$ represents any collection of partial derivatives of
order less than or equal to $k$, with $k\in \{ 1,2,3,4\}$.

\begin{thm}\label{main-thm} Let $(M,g)$ be a Riemannian manifold of 
dimension $n\geq 3$ with one asymptotically flat end $M \backslash K$, 
parameterized as in \eqref{end-parameterization1} and 
\eqref{metric-expansion1}. Then 
there exists $\rho_0>1$ and compact sets $K_\rho$ indexed by $\rho
\in (\rho_0, \infty)$ such that the domains $\Omega_{\rho} = M
\backslash K_\rho$ are critical capacitors. In other words, there
exist functions $\bar u_\rho$ which solve the over-determined
boundary value problem
\begin{align}\label{eq:capacitor1}
\begin{cases}
\Delta_g \bar u_\rho  = 0& \quad  \textrm { in } \quad\Omega_\rho\vspace{3mm} \\
 \bar u_\rho  =  1& \quad\textrm{ on } \quad\partial \Omega_\rho =
\partial K_\rho\vspace{3mm}\\
 \lim_{|y| \rightarrow \infty} \bar u_\rho(\phi (y))  =  0&
\vspace{3mm}\\
  \dfrac{\partial \bar u_\rho} {\partial \eta}
= C(\rho, \sigma, n)&\quad
\textrm { on } \quad\partial \Omega_\rho =
\partial K_\rho,
\end{cases}
 \end{align} where  $\eta$ is the unit interior normal to $K_\rho$ and 
$$
C(\rho, \sigma, n)=\frac{n-2}{\rho}+\dfrac{(n-2)(n-3)}{2\rho^{n}} \sigma.
$$
The hypersurfaces $\{ \partial K_\rho\}_{\rho>\rho_0}$
foliate the $M \backslash K_{\rho_0}$.
\end{thm}

Our result builds naturally on earlier work, particularly that of 
the first and second authors \cite{FM}. 
 More precisely, they perturb small geodesic balls to
 produce a family of domains $\Omega_\rho$, parameterized by $\rho
 \in (0, \rho_0)$, which admit solutions to the overdetermined boundary value
 problem
 $$\Delta_g u = 1 \textrm { on }\Omega_\rho, \qquad
 \left. u \right |_{\partial \Omega_\rho} = 0, \qquad
 \left. \frac{\partial u}{\partial \eta} \right |_{\partial \Omega} = \textrm{ constant.}$$
In our case, the sets $K_\rho$ will be perturbations of large coordinate spheres, 
as defined by the parameterization $\Phi$ in \eqref{end-parameterization1}.

We end this introduction with a brief outline of the 
rest of the paper. We begin by reformulating our problem 
in Section \ref{sec:reformulation} to take place on a fixed set. 
We parameterize this reformulated problem by a radius $\rho$, 
a translation $\tau$, and a function $w \in \mathcal{C}^{2,\alpha} (\Ss)$. 
Section \ref{sec:preliminary} has some preliminary 
computations, such as expansions of the metric and the Laplace-Beltrami 
operators for our reformulated problem, as well as a study of 
the mapping properties of the Laplace-Beltrami operator on 
certain weighted function spaces in Section \ref{sec:dirichlet}. 
In Section \ref{sec:soln_construction} we construct an approximate
solution $v$, given in \eqref{eq:vv}, and perturb it by a translation to the 
eventual solution $\widehat{u}_{\rho,\tau,w}$, given in \eqref{eq:u-r-w-t}. 
The function $\widehat{u}_{\rho,\tau,w}$ already satisfies most of our desired 
properties: it is harmonic, decays appropriately, and has constant Dirichlet 
data. This it only remains for us to choose parameters $\rho$, $\tau$, 
and $w$ so that $\widehat{u}_{\rho,\tau,w}$ also has constant 
Neumann data. To correctly choose these parameters we 
must invert the Dirichlet-to-Neumann operator of the Laplace-Beltrami 
operator. We do this in two steps, first writing out an expansion of the 
normal derivative of $\widehat{u}_{\rho,\tau,w}$ and performing a 
linear analysis of this expansion in Section \ref{sec:linear}, and then 
completing our nonlinear analysis using the implicit function theorem 
in Section \ref{sec:nonlinear}. 
Finally, in Section \ref{sec:foliation} we show that we do in fact
produce a foliation of the asymptotically flat end.

\bigskip
\noindent \textbf{Acknowledgements}: M. M. F. is partially supported by Alexander van Humboldt Foundation, 
J. R. was partially supported by the National Research Foundation of South Africa, and 
I. A. M. was partially supported by the Abbas Bahri Excellence Fellowship. We completed part of this research 
when M. M. F. visited the University of Cape Town, when J. R. visited the African Institute for Mathematical Sciences
in Mbour, Senegal, and when I. A. M. visited the Department of Mathematics at the University of Rutgers. We 
thank all these institutions for their hospitality. We also thank Murray Christian for enlightening 
conversations when we began this project. 

\section{Reformulation of the problem}\label{sec:reformulation}

In this section we reformulate our problem so that we can solve a
family of PDEs on the fixed Euclidean domain $\R^n \backslash \B$. 
Intuitively, we accomplish three things with this reformulation. First, 
we rescale by $\rho> 0$, which one should take to be large. Second, 
we translate the center of the ball by a small parameter $\tau \in \R^n$. 
Third, we deform the unit sphere $\Ss = \partial \B$ by a function 
$w \in \mathcal{C}^{1,\alpha} (\Ss)$. Putting all these transformations 
together we obtain a parameterization 
\begin{equation}\label{eq:param}
\Phi_{\rho,w,\tau}: \R^n \backslash \B \rightarrow M, \qquad \Phi_{\rho,w,\tau} (x)
= \Phi\left (\rho\tau + \rho x + \rho w \left ( \frac{x}{|x|} \right ) x  \right ) 
\end{equation}
and let $\widehat{\Omega}_{\rho,\tau,w} = \Phi_{\rho,\tau,w} (\R 
\backslash \B)$. 
 Finally we can use $\Phi_{\rho,w,\tau }$ to pull problem \eqref{eq:capacitor1}
back to $\R^n \backslash \B$. Under this change of coordinates, we
have now reformulated our original problem \eqref{eq:capacitor1} as
\begin{align}\label{eq:extremal-citor2}
\begin{cases}
\Delta_{\widehat g} \widehat u_{\rho,w,\tau} =0& \quad \textrm { in } \quad\R^n
\backslash \B \vspace{3mm}\\
  \widehat u_{\rho,w,\tau }  =  1&\quad
\textrm{ on
}\quad \partial \B \vspace{3mm}\\
 \lim_{|x| \rightarrow \infty}
\widehat u_{\rho,w,\tau } (x) = 0& \vspace{3mm}\\
 \dfrac{\partial \widehat
u_{\rho,w,\tau } } {\partial \widehat \eta}  = Const.& \quad\textrm { on }\quad\partial
\B,
\end{cases}
\end{align}where
$\widehat{g} = \Phi^*_{\rho,w,\tau}(g)$, $\widehat \eta$ is the inward pointing 
unit normal to $\Ss = \partial \B$ with respect to the metric $\widehat g$. Both our new metric 
$\widehat g$ and the function $\widehat{u}_{\rho,w, \tau}$ depend on the three parameters
 $\rho \in (0,\infty)$, $w \in \mathcal{C}^{2,\alpha} (\Ss)$, and $\tau \in \R^n$. We 
 should imagine $\rho$ to be large and both $w$ and $\tau$ to be 
 small. So that our parameters are all of the same scale, we require 
 \begin {equation} \label{param-relations}
\| w \|_{\mathcal{C}^{2,\alpha}(\Ss)} = \mathcal{O}(\rho^{-n}), \qquad 
|\tau| = \mathcal{O}(\rho^{-n}) \end {equation}
for the remainder of the paper. 

\section{Preliminary computations} \label{sec:preliminary}
In this section we carry out some preliminary computations, in
preparation for solving \eqref{eq:extremal-citor2} We first write
out a Taylor expansion of the metric $\widehat g$.  \\

\subsection{Notation} 

All our computations in this section are perturbation 
expansions, when $\rho$ is large and 
$|\tau|$ and $\| w \|_{\mathcal{C}^{k,\alpha} (\Ss)}$ 
are small. In the computations below we will sometimes 
wish to extend a function $v$ defined on the sphere to a tubular 
neighborhood, and we do so by making it constant in 
the radial direction, taking $w(x) = w\left ( \frac{x}{|x|} \right )$. 
Similarly, for each $w \in \mathcal{C}^{2,\alpha} ( \Ss )$ 
we let 
\be
\label{eq:defwi-wij}
w_i\left(\frac{x}{|x|}\right)=\de_iw\left(\frac{x}{|x|}\right),
\qquad w_{ij}\left(\frac{x}{|x|}\right)=\de_{ij}
w\left(\frac{x}{|x|}\right) \ee and \be \Delta_0 w= \sum_{i=1}^n w_{ii},
\ee { for all $x\in \R^n\setminus\{0\}$}.

To make the computations below tractable, we adopt the 
following notation throughout the rest of the paper. 

For $i \in \{ 0,1,2\}$ we let $L^i$ denote a linear partial 
differential operator of order $i$ whose coefficients 
depend smoothly on $\rho$ and $x$ and that satisfies 
the bound 
\begin{equation}\label{eq:estiL}
\| L^i(v)|\|_{C^{k,\alpha}(\R^{n}\setminus\B)}\leq
c\|v\|_{C^{k+i,\alpha}(\Ss )}
\end{equation}
for each $v\in
C^{k,\alpha}(\mathbf{S})$, where 
the constant $c>0$ is independent of $\rho$. Similarly we let 
$Q^i$ denote a nonlinear operator of order $i \in \{0,1,2\}$ such 
that $Q^i (0,0) = 0$ and that satisfies the bound 
\begin{eqnarray*} 
\| Q^i(v_1, \tau_1) - Q^i (v_2, \tau_2) \|_{\mathcal{C}^{k, \alpha}(\R \setminus \B)}
 & \leq & c\left ( \| v_1 \|_{\mathcal{C}^{k+i, \alpha}(\Ss)} + \|v_2\|_{\mathcal{C}^{k+i, \alpha}
 (\Ss)} + |\tau_1| + |\tau_2| \right ) \\ 
 & & \quad \times \left ( \| v_1-v_2\|_{\mathcal{C}^{k+i, \alpha}(\Ss)} + |\tau_1 - \tau_2 |
 \right ) \end {eqnarray*}
   provided
$\|v_1\|_{C^{k+i,\alpha}(\Ss  )}+\|v_2\|_{C^{k+i,\alpha}(\Ss
)}+|\tau_1|+|\tau_2|\leq 1$.

Finally we let $P_i$ be a function of the form 
\be \label{eq:def-of-P-i} P_i(\rho,x, v, \tau)=\rho^{1-n}
|x|^{1-n} L^i(v)+Q^i(v,\tau)+\cO(\rho^{-n}|x|^{-n}) \ee 
such that for every $k,\ell,m\in \N$, $x\in \R^n\setminus\B$, $\tau\in
\R^n$ and $\rho>\rho_0>0$ 
\be\label{eq:def-Pi} \|\de_\rho^\ell
\de_v^k \de_\tau^m P_i(\rho, \cdot , v,\tau )  \| \leq c \left(
\rho^{-n-\ell}+ \rho^{1-n- \ell}( \|v \|_{\mathcal{C}^{i,\alpha}(
\Ss )}) + \|v \|_{\mathcal{C}^{i,\alpha}( \Ss )}^2 +| \tau|^2
\right)  , \ee 
for some positive constant depending only on
$h,n,k,\ell,m,i, \alpha$ and $\rho_0$.  For brevity we 
write 
$$
P_i(\rho,x,v)=P_i(\rho,x, v, 0).
$$
It is important to observe that the product of any two terms, 
each of which has the form of either $L^i$ or $Q^i$, has the 
form of $P_i$. 

\subsection{Metric expansions and the Laplacian} 

We have the following
expansions.
\begin{lem} \label{eq:metric} We have\\
\begin{eqnarray}\label{metric-expansion2}
\rho^{-2}\widehat g_{ij} (x)  & = & (1+2w+\sigma  |z|^{1-n})\delta_{ij} + x_i  w_j + x_j  w_i+ P_1(\rho,x,w,\tau)
\end{eqnarray}
and
\begin{eqnarray}\label{metric-expansion3}
 \rho^{2}\widehat g^{ij} (x) & = &(1+2w+\sigma  |z|^{1-n})^{-1} \delta_{ij} - ( x^i w^j + x^j w^i )+
P_1(\rho,x,w,\tau),
\end{eqnarray}
where 
\begin{equation}\label{eqz}
|z|^{1-n}:=\rho^{1-n} r^{1-n} \left ( 1-
\frac{n-1}{r^2} \langle x, \tau \rangle  \right).
\end{equation}
\end{lem}

\begin {proof} 
%Recall that 
%\begin{equation}\label{eq: metri1}
% g_{ij}(y) = D\Phi(y) (e_i) D\Phi(y) (e_j) = (1+ \sigma |y|^{1-n}) \delta_{ij} + h_{ij}(y) .
%\end{equation}
Letting $\{ e_1, \dots, e_n\}$ be the standard orthonormal basis for $\R^n$, we see  
\begin{eqnarray*} 
D\Phi_{\rho,w,\tau} (e_i)  & = &  \sum_k D\Phi(e_k) \left ( \left ( \rho + \rho w \left ( \frac{x}{|x|} \right ) 
\right ) \delta_{ik} + \rho x^k w_k \right ) \\ 
& = & \rho \left ( \left ( 1 + w \left ( \frac{x}{|x|} \right ) \right ) D\Phi(e_i) + \sum_k x^k w_k, D\Phi(e_k) 
\right ), 
\end{eqnarray*} 
where we evaluate derivatives of $\Phi$ at $\rho \tau + \rho x + \rho w(x/|x|) x$. 
Hence
\begin{eqnarray} \label{eq:metric2} 
\rho^{-2} \widehat g_{ij} &= &\left ( \left ( 1+w\left ( \frac{x}{|x|} \right ) \right ) D\Phi(e_i) + \sum_k 
x^k w_k D\Phi(e_k) \right )  \\ \nonumber 
&& \times \left ( \left ( 1+ w\left ( \frac{x}{|x|} \right ) \right ) D\Phi(e_j) + \sum_l 
x^l w_l D\Phi(e_l) \right ) \\ \nonumber 
& = & (1+w) g_{ij} + (1+w) w_j x^l g_{il} + (1+w) w_i x^k g_{jk} + w_i w_j x^k x^l g_{kl} .
\end{eqnarray} 
%and \eqref{metric-expansion2} now follows from \eqref{metric-expansion1} and 
%\eqref{eq:metric2}. 
Next, we  write $|x| = r$ and 
$$y = \rho \left ( 1+ w \left ( \frac{x}{|x|} \right ) \right ) x + \rho \tau,$$
so that
\begin {align} \label {expnormy}
|y|^2   = & \rho^2 r^2 \left ( 1+ 2w + \frac{2}{r^2} \langle x,
\tau \rangle  +w^2+  2r^{-2} w \langle x, \tau
\rangle  + r^{-2}\mathcal{|\tau|}^2 \right ) \\ \nonumber
|y|^{1-n} &=\rho^{1-n} r^{1-n} \left ( 1-(n-1) w -
\frac{n-1}{r^2} \langle x, \tau \rangle   + w^2+
r^{-2} w \langle x, \tau \rangle  +
r^{-2}\mathcal{|\tau|}^2 + \cdots  \right)\\ \nonumber
& =\rho^{1-n} r^{1-n} \left ( 1-
\frac{n-1}{r^2} \langle x, \tau \rangle  \right)+\rho^{1-n} r^{1-n}L^0(w,\tau)+ Q^0(w,\tau) . \\ \nonumber
\end {align}
Observe that we absorb the term $(1-n) \rho^{1-n} r^{1-n} w$ above into $L^0(w)$, while the corresponding linear term  with respect to  $\tau$ is kept. Indeed 
when solving the nonlinear equation for  small $\rho$ in Section \ref{sec:nonlinear}, we  have to replace  $w$  with $\rho^{n-1} w$ , which increases the power of $\rho$ in $(1-n) \rho^{n-1} r^{1-n} w$ by  $n-1$. 

Using  \eqref{metric-expansion1} and  \eqref{eq:defwi-wij}, it follows from \eqref{eq:metric2} 
and  \eqref{expnormy} that  both \eqref{metric-expansion2} and \eqref{metric-expansion3} hold.
\end {proof}
In the next sections, we will work with the metric \begin{equation}\label{eq:chanageof metric}
g_{\rho}:=\rho^{-2}\hat{g}. 
\end{equation}

\begin{lem}\label{cor:exp-Lap-hat-g}
For $\rho$ sufficiently large, 
\begin{align} 
 \Delta_{g_{\rho}} & = 
(1-2w-\sigma |z|^{1-n})\Delta_{0}
 - \biggl [  x_i w_j
 + x_j w_i + P_1(\rho,x,w,\tau) \biggl] \partial^{2}_{ij}\nonumber\\
 &-\biggl [ \biggl(3w_j- \frac{n-2}{2} \sigma  \partial_j |z|^{1-n} \biggl)+x_iw_{ij}+[\Delta_{0}w]
x_j+ P_2(\rho,x,w,\tau)\biggl]\partial_j, \nonumber
\end{align}
where $\Delta_0$ is the usual flat Laplacian.

Moreover, 
$$\partial_j |z|^{1-n}=-(n-1)\rho^{1-n} r^{-1-n} x_j\left ( 1 -
\frac{n-3}{r^2} \langle x, \tau \rangle \right)\\
-(n-1)\rho^{1-n} r^{-1-n} \tau_j. $$
\end{lem}

\begin {proof} 
Recall that the Laplace-Beltrami operator has the form 
$$\Delta_{g_{\rho}} u = \frac{1}{\sqrt {|g_{\rho}|}} \partial_i g_{\rho}^{ij} \sqrt{ |g_{\rho}|}
\partial_j u, \qquad |g_{\rho}| = \det (g_{\rho}). $$

Then using  \eqref{metric-expansion2} and \eqref{metric-expansion3}, we have
\begin {eqnarray*} 
|g_{\rho}| & = &1+2nw+ n\sigma  |z|^{1-n} + P_1(\rho,x,w,\tau)\\ 
\sqrt{|g_{\rho}|} & = &1+nw+ \frac{n\sigma}{2} |z|^{1-n}+ P_1(\rho,x,w,\tau) \\ 
\frac{1}{\sqrt{|g_{\rho}|}} & = &1-nw-\frac{n\sigma}{2} |z|^{1-n}+ P_1(\rho,x,w,\tau)\\
g_{\rho}^{ij} \sqrt{|g_{\rho}|}&  = & \left [ \left ( 1+ (n-2) w +  \frac{n-2}{2} \sigma  |z|^{1-n} \right ) \delta_{ij} - \left ( w^ix^j+  w^jx^i \right ) + P_1(\rho,x,w, \tau)  \right ] \\
\end {eqnarray*}

\begin {eqnarray*}
\partial_i (g_{\rho}^{ij} \sqrt{|g_{\rho}|}) & = &  
\biggl((n-3)w_i + \frac{n-2}{2} \sigma  \partial_i |z|^{1-n}\biggl)\delta_{ij} - w_{ii} x_j  - w_{ij} x_i-w_j   + P_2 (\rho,x,w, \tau)  \\ 
\frac{1}{\sqrt{|\tilde{g}|}}\partial_i (\tilde{g}^{ij} \sqrt{|\tilde{g}|}) & = &  
\biggl((n-3)w_i + \frac{n-2}{2} \sigma  \partial_i |z|^{1-n}\biggl)\delta_{ij} - w_{ii} x_j  - w_{ij} x_i-w_j   + P_2 (\rho,x,w, \tau)  \\ 
\end {eqnarray*}

\begin {eqnarray*}
\Delta_{g_{\rho}} u & = & g_{\rho}^{ij}  \partial_i \partial_j u + \frac{1}{\sqrt{|g_{\rho}|}} 
\partial_i (g_{\rho}^{ij} \sqrt{|g_{\rho}|}) \partial_j u \\ 
& = & \left [ (1-2w - \sigma |z|^{1-n}) \delta_{ij} - (x^i w^j + x^j w^i) 
+ P_1(\rho,x,w,\tau)\right ] \partial_i \partial_j u \\ 
&-&\left [ \biggl(3w_i- \frac{n-2}{2} \sigma  \partial_i |z|^{1-n} \biggl)\partial_i u
- (x^j \Delta_0 w +x_i w_{ij})\partial_j u\right ] +  P_2(\rho, x, w, \tau)\partial_j u.
\end {eqnarray*} 
In addition we have from \eqref{eqz}
\begin {align}\partial_j |z|^{1-n}=&-(n-1)\rho^{1-n} r^{-1-n} x_j\left ( 1 -
\frac{n-3}{r^2} \langle x, \tau \rangle \right)-(n-1)\rho^{1-n} r^{-1-n} \tau_j,
% \\ \nonumber
%g_{ij}(y) & = & (1+\sigma \rho^{1-n} r^{1-n} ) \delta_{ij} + \mathcal{O}(w^2) +
%\mathcal{O}(r^{-n}).
\end {align}
which yield the expansions in Lemma \ref{cor:exp-Lap-hat-g}. 
\end {proof} 

\subsection{Weighted spaces } \label{sec:dirichlet}

The best setting in which to perform our linear analysis is that of
weighted H\"older spaces. Following Pacard and Rivi\`ere, we use the
following definition.
\begin {defn}
Let $\nu \in \R$, $k \in \mathbb{N}$ and $0 < \alpha < 1$. Then we say $u \in
\mathcal{C}^{k,\alpha}_\nu( \R^n \backslash \B)$ if $u \in
\mathcal{C}^{k,\alpha}_{\textrm{loc}}(\R^n \backslash \B)$ and
$$\| u \|_{k,\alpha,\nu} = \sup_{s > 1} \left ( s^{-\nu} [ u ]_{k,\alpha, s}  \right ) < \infty.$$ Here
$$ [ u ]_{k,\alpha, s} := \sum_{i = o}^k s^i \sup_{A_s} |\nabla^i u| + s^{k + \alpha}\sup_{x,x' \in A_s}\frac{|\nabla^k u(x) - \nabla^k u(x')|}{|x - x'|^{\alpha}}, $$ where
$A_s = \{ x \in \R^n : s< |x|< 2s\}$.
We denote the space of functions vanishing on the boundary by
$$ \mathcal{C}^{k,\alpha}_{\nu, \mathcal{D}}( \R^n \backslash \B) := \left \{ u \in \mathcal{C}^{k,\alpha}_\nu( \R^n \backslash \B): u|_{\partial B} = 0 \right \}. $$
\end {defn}
Intuitively, one can think of $\mathcal{C}^{0,\alpha}_\nu (\R^n
\backslash \B)$ as those functions which grow at most like $|x|^\nu$
when $|x|$ is large.

\begin {rmk}Pacard and Rivi\`ere perform their analysis on weighted H\"older spaces on $\B\backslash \{ 0 \}$, whereas we want to examine functions on $\R^n \backslash \B$.
It is straight-forward to transfer between the two settings using the Kelvin
transform $\mathbb{K}$, defined by
\begin {equation} \label {defn-kelvin-trans}
\mathbb{K}: \mathcal{C}^{k,\alpha}_\nu (\B \backslash \{ 0 \} ) \rightarrow
\mathcal{C}^{k,\alpha}_{2 -n -\nu} (\R^n \backslash \B), \quad
\mathbb{K} (u) (x) = |x|^{2-n} u \left ( \frac{x}{|x|^2} \right ) .\end {equation}
It will be convenient to also note the transformation law
\begin {equation} \label{kelvin-trans-law}
\Delta_0 \left ( \mathbb{K}(u) \right ) (x) = |x|^{-4} \mathbb{K} (\Delta_0 u)
(x). \end {equation}
\end {rmk}

One can show the following theorem (see Section 2.2 of \cite {PR}).
\begin {thm}
\label{th::inv-D_0}
 The mapping
$$\Delta_0 : \mathcal{C}^{2,\alpha}_{\nu, \mathcal{D}}(\R^n \backslash \B) \rightarrow
\mathcal{C}^{0,\alpha}_{\nu-2}(\R^n \backslash \B)$$ is injective if
$\nu < 0$ and surjective if $\nu > 2-n$.
\end {thm}

The mapping properties of $\Delta_0$ change when the weight $\nu$
crosses over one of the indicial roots $\gamma_j^\pm$, where
$$\gamma_j^\pm = \frac{2-n}{2} \pm \sqrt{ \frac{(n-2)^2}{4} + \lambda_j}.$$
and $\lambda_j$ is  the $j$th eigenvalue of the Laplace-Beltrami
operator on the sphere. Thus one can recover the indicial roots
$\gamma_j^\pm$ as growth/decay rates of solutions to the ODE
$$w'' + \frac{n-1}{r} w' - \frac{\lambda_j}{r^2} w = 0, \qquad
    w= r^{\gamma_j^\pm} .$$

A slightly more refined analysis uncovers the following theorem.
\begin {thm} Let $\nu < 0$ with $\nu \not \in \{ \gamma_j^\pm : j \in \mathbf{N}
\}$, and let $j_0$ be the least non-negative integer such that $\nu
> \gamma_{j_0}^-$. Then the cokernel of the mapping
$$\Delta_0 : \mathcal{C}^{2,\alpha}_{\nu, \mathcal{D}} (\R^n \backslash \B) \rightarrow
\mathcal{C}^{0,\alpha}_{\nu-2} (\R^n \backslash \B)$$ has dimension
$j_0$. Alternatively let $\nu > 2-n$ with $\nu \not \in \{
\gamma_j^\pm : j \in \mathbf{N} \}$, and let $j_0$ be the least
positive integer such that $\nu < \gamma_{j_0}^+$ then kernel of the
mapping
$$\Delta_0 : \mathcal{C}^{2,\alpha}_{\nu, \mathcal{D}} (\R^n \backslash \B) \rightarrow
\mathcal{C}^{0,\alpha}_{\nu-2} (\R^n \backslash \B)$$ has dimension
$j_0$. \end {thm}
Again, we refer the reader to Section 2.2 of
\cite{PR} for details.\\

Replacing $\rho$ by $1/\rho$ in Lemmas \ref{eq:metric}- \ref{cor:exp-Lap-hat-g}  and keeping the notation  $g_{\rho}$ for the metric $g_{1/\rho}$, we have the following result. 
 \begin{lem}\label{lem:Oper-inv}
 There exist $\rho_0>0$ and $c_0>0$ such that  the map $$\cL: (0,\rho_0)\times \B_{c_0}(0)\times  
 \B_{c_0}(0)\times \mathcal{C}^{2,\alpha}_{\nu, \mathcal{D}}(\R^n \backslash \B) \to  
 \mathcal{C}^{0,\alpha}_{\nu }(\R^n \backslash \B) $$ defined by
$$
\cL(\rho, w,\tau, u):=\Delta_{g_{\rho}} u
$$
is well defined and smooth.

Furthermore, for every $(\rho, w,\tau)\in (0,\rho_0) \times
(\B_{c_0}(0))^2$, the linear map
$$\cL_{\rho,w,\tau}: \mathcal{C}^{2,\alpha}_{\nu, \mathcal{D}}(\R^n \backslash \B) \to \mathcal{C}^{0,\alpha}_{\nu }(\R^n \backslash
\B), \quad u\mapsto\cL_{\rho,w,\tau}(u):= \cL(\rho, w,\tau,  u )$$ is
invertible
 and for all $u\in \mathcal{C}^{2,\alpha}_{\nu, \mathcal{D}}(\R^n
\backslash \B)$ we have the inequalities 
\begin{equation}\label{eq:invertcomp}
C\| u\|_{ \mathcal{C}^{2,\alpha}_{\nu }(\R^n \backslash \B) }\leq
\|\cL_{\rho,w,\tau} (u)\|_{ \mathcal{C}^{0,\alpha}_{\nu }(\R^n
\backslash \B) }\leq C^{-1} \| u\|_{ \mathcal{C}^{2,\alpha}_{\nu
}(\R^n \backslash \B) },
\end{equation}
where $C>0$ is independent of $(\rho,w, \tau).$
 \end{lem}

\begin {proof}
By  Lemma \ref{cor:exp-Lap-hat-g}, we can find $\rho_0>0$ and
$c_0>0$ such that  the map $$\cL: (0,\rho_0)\times
\B_{c_0}(0)\times \B_{c_0}(0)\times  \mathcal{C}^{2,\alpha}_{\nu,
\mathcal{D}}(\R^n \backslash \B) \to  \mathcal{C}^{0,\alpha}_{\nu
}(\R^n \backslash \B) $$ defined by
$$
\cL(\rho, w,\tau, u)=\Delta_{g_{\rho}} u
$$
is well defined and smooth. We shall show that  the linear
    map
$$u\mapsto  \cL(\rho, w,\tau, u ): \mathcal{C}^{2,\alpha}_{\nu, \mathcal{D}}(\R^n \backslash \B) \to \mathcal{C}^{0,\alpha}_{\nu }(\R^n \backslash \B) $$
is invertible for every $(\rho, w,\tau)\in (0,\rho_0)\times
\B_{c_0}(0)$. To see this, we pick $f\in
\mathcal{C}^{0,\alpha}_{\nu }(\R^n \backslash \B) $ and we define
$$
\cF : (-\rho_0,\rho_0)\times
\B_{c_0}(0)\times\B_{c_0}(0)\times \mathcal{C}^{2,\alpha}_{\nu,
\mathcal{D}}(\R^n \backslash \B) \to \mathcal{C}^{0,\alpha}_{\nu
}(\R^n \backslash \B)
$$
by
$$
\cF(\epsilon,w,\tau,u):=  \cL(|\epsilon|, w,\tau, u )-f.
$$
It is clear from Lemma \ref{cor:exp-Lap-hat-g} that $\cF$ is of
class $C^2$, since $n\geq 3$.  Let $u_0\in
\mathcal{C}^{2,\alpha}_{\nu, \mathcal{D}}(\R^n \backslash \B)$ be
the unique solution to $\Delta_0 u_0=f$.  We have $\cF(0,0, 0, u_0)=0$
and by Theorem \ref{th::inv-D_0}, $\de_u \cF(0,0,0,u_0)$ is
invertible. By the implicit function theorem, there exists a unique
$u_{\epsilon,w,\tau}$ satisfying $ \cF(\epsilon,w,\tau, u_{\epsilon,w,\tau})=0$, with
$u_{0,0,0}=u_0$. We then conclude that, provided $c_0$  and 
$\rho_0$   small, the linear map
$$
  \cL(\rho, w,\tau \cdot ): \mathcal{C}^{2,\alpha}_{\nu, \mathcal{D}}(\R^n \backslash \B) \to \mathcal{C}^{0,\alpha}_{\nu }(\R^n \backslash \B) $$
is invertible for every $ \rho,w,\tau\in (0, \rho_0)\times
(\B_{c_0}(0))^2$. The bound \eqref{eq:invertcomp} also follows from the implicit function theorem. 
\end {proof}

\section{Approximate and actual solutions}
\label{sec:soln_construction} 

In this section we construct an approximate solution using the 
standard Greens function in Euclidean space and compare it to the 
solution of a corresponding Dirichlet problem.\\

For $w\in\B_{c_0}(0)$, we define
\begin{equation}\label{eq:vv}
v(x) = v_{\rho,w} (x) = \left | \left ( 1 + w \left ( \frac{x}{|x|} \right )
\right ) x \right |^{2-n}.
\end{equation}
We have the following expansion.
\begin {lem}\label{eq:lapv}
For $\rho$ sufficiently small the  Laplacian of $v$ is given by 
\begin {eqnarray} \label{lap-expansion2} 
\Delta_{g_{\rho}} v & = &  
\frac{(n-1)(n-2)^2}{2} \sigma \rho^{n-1} r^{1-2n} \\ \nonumber 
&& -\frac{(n-1)(n-4)(n-2)^2}{2} \sigma \rho^{n-1} r^{1-2n} \left \langle  \frac{x}{r^2}, \tau \right \rangle 
+ P_2(1/ \rho,x,w,\tau) .
\end {eqnarray} 
\end {lem} 

\begin {proof}
By definition 
\begin{equation}%\label{eq:1pluswto-alph}
v(x)=|x|^{2-n}\biggl(1-(n-2)w +Q^{0}(w,\tau)\biggl) ,
\end{equation}
and so 
\begin{eqnarray} \label{eq:1plaplv}
\frac{\partial v}{\partial x_i}&=&-(n-2)r^{-n}\biggl(1-(n-2)w+Q^{0}(w,\tau)\biggl)x_i\\ \nonumber 
&-&(n-2)r^{2-n}w_i+Q^{1}(w,\tau) \\ \nonumber
\partial_{ij}^{2} v&=&-(n-2)r^{-n}\biggl(1-(n-2)w\biggl)\delta_{ij}\\ \nonumber 
&+&n(n-2)r^{-2-n}\biggl(1-(n-2)w\biggl)x_ix_j+
(n-2)^{2}r^{-n}w_jx_i\\ \nonumber 
&+& (n-2)^{2}r^{-n}w_ix_j-(n-2)r^{2-n}w_{ij}+Q^{2}(w,\tau)\\ \nonumber 
 \Delta_0(v)&=-&(n-2)r^{2-n}\Delta_0w+Q^{2}(w,\tau).
 \end{eqnarray}

With this, we have
 \begin{align}
 &(1-2w-\sigma |z|^{1-n})\Delta_{0}(v)=-(n-2)r^{2-n}\Delta_0w+Q^{2}(w,\tau)\\
 &\biggl [  x_i w_j
 + x_j w_i + P_1(\rho,x,w,\tau) \biggl] \partial^{2}_{ij}(v)=P_2(\rho,x,w,\tau).\nonumber\\
-&\biggl [3w_j+x_iw_{ij}+[\Delta_{0}w]x_j+ P_2(\rho,x,w,\tau]\biggl]\partial_j(v)=(n-2)r^{2-n}\Delta_0w+(n-2)r^{-n}x_ix_jw_{ij}\nonumber\\
& \frac{n-2}{2} \sigma  \partial_j |y|^{1-n} \partial_j(v)=\frac{(n-1)(n-2)^2}{2} \sigma \rho^{1-n} r^{1-2n}-\frac{(n-1)(n-4)(n-2)^2}{2} \sigma \rho^{1-n} r^{1-2n}\langle  \frac{x}{r^2}, \tau \rangle.
\end{align}
However, 
$ 0 = \partial_i( x^j w_j) = \delta_{ij}w_j + x^j w_{ij} \Rightarrow x^i x^j w_{ij} 
= - x^i \delta_{ij}j w_j = - x^i w_i = 0.$ 
The expansion in Lemma \ref{eq:lapv} now follows from Lemma \ref{cor:exp-Lap-hat-g} after  replacing $\rho$ by $1/\rho.$
\end {proof}

%\subsection{Solution to the  Dirichlet problem}
Using Lemma \ref{lem:Oper-inv},
%\ref{linear-solvability}, Lemma \ref{approxdel-soln1},
%\eqref{eq:1pluswto-alph} and \eqref{eq solaapp}
 we construct a unique solution $\Psi_{\rho,w} (x)$  to the equation
\begin{equation}\label{eq:extremal-capacitor2r}
\begin{cases}
\Delta_{g_{\rho}}  \Psi_{\rho,\tau, w} &= -\Delta_{g_{\rho}} v \quad
\textrm { in } \quad\R^n
\backslash \B \vspace{4mm}\\
\Psi_{\rho,w}&=1-v= (n-2)w +Q^0(w)  \quad \textrm{ on }\quad
\partial \B.
% \vspace{3mm}
\end{cases}
\end{equation}
First choose $\chi \in \mathcal{C}^\infty_c(\R_+)$ such that 
$\chi(t) = 1$ for $1/2 \leq t \leq 1$, and define 
$$f (\rho,x,w) = -\Delta_{g_{\rho}}\biggl(
((n-2)w +Q^0(w))\chi(|x|)\biggr)-\Delta_{g_{\rho}} v 
\in \mathcal{C}^{0,\alpha}_\nu (\R^n \backslash \B)$$
for any $\nu \in (2-n, 0)$.  
Next we use Lemma \ref{lem:Oper-inv} to let $\widetilde{\Psi}_{\rho,\tau,w}$ 
be the unique solution of 
\begin{equation} \label{eq:extremal-capacitor21}
\Delta_{g_\rho} \widetilde{\Psi}_{\rho,\tau,w} (x)= f(\rho,x,w) \textrm{ for } x \in 
\R^n \backslash \B, \qquad \widetilde{\Psi}_{\rho,\tau,w} = 0 \textrm{ on } 
\partial \B. 
\end{equation} 
Observe that  $f$ depends smoothly on  $(\rho,w,\tau)$ and that the
mapping $A\mapsto (A)^{-1}$ is smooth in the open set of linear
invertible operator between two Banach spaces. We   deduce that  the
mapping $$ (\rho,\tau, w)\mapsto  \widetilde{\Psi}_{\rho,\tau, w}$$ is smooth. We then
have
\begin{equation}\label{eq:psi}
 \Psi_{\rho,\tau, w}=  \widetilde{\Psi}_{\rho,\tau , w}+  ((n-2)w +Q^0(w))\chi(|x|),
 \end{equation}
which solves uniquely \eqref{eq:extremal-capacitor2r}.\\

Finally we define 
\begin{equation}\label{eq:u-r-w-t}
\widehat{u}_{\rho,\tau, w}(x)=v (x)+\Psi_{\rho,\tau, w} (x),
 \end{equation}
 which satisfies 
 \begin{equation}\label{eq:extremal-capacitor2}
\begin{cases}
\Delta_{g_{\rho}}  \widehat{u}_{\rho,\tau, w} =0 \quad \textrm { in } \quad\R^n
\backslash \B \vspace{4mm}\\
\widehat{u}_{\rho,\tau,w}=  1  \quad \textrm{ on }\quad \partial \B.
% \vspace{3mm}
\end{cases}
\end{equation}
Since  $f(\rho,\cdot,w) \in  \mathcal{C}^{0,\alpha}_{\nu}(\R^n
\backslash \B)$, \eqref{eq:extremal-capacitor21}  and   \eqref{eq:invertcomp} 
imply that $\widetilde{\Psi}_{\rho,w}(x)\rightarrow 0 \quad \textrm{as}\quad |x|\rightarrow \infty$. 
We conclude with  \eqref{eq:vv},  \eqref{eq:psi} and  \eqref{eq:u-r-w-t} that   
$$\widehat{u}_{\rho,\tau, w}(x)\rightarrow 0 \quad \textrm{as}\quad |x|\rightarrow \infty.$$

\section{Expansion of the normal derivative and linear analysis} 
\label{sec:linear} 

The aim of this section is to derive an expansion of the normal derivative of the 
solution  $ \widehat{u}_{\rho,\tau, w}$. We start by the computation of the interior unit normal vector to $\mathbf{B}$ .
\begin{lem}\label{lem Nordd}
Let $\Theta: \mathbb{R}^{n-1}\longrightarrow \Ss$ parameterize $\Ss$ by the inverse of 
stereographic projection and let 
$$\Theta_k:=\partial_{k}\Theta\quad k=1,...,n-1.$$
 The interior unit normal vector field to
$\mathbf{B}$  with respect to the metric $g_{\rho}$ is given by
\begin{equation}\label{eq:vectuni-in-lem}
\widehat{\nu}_{g_{\rho}}=\frac{\widehat{\eta}}{\sqrt{ \langle \widehat{\eta},
\widehat{\eta}\rangle_{g_{\rho}}
}}=\biggl[1-w-\frac{\sigma}{2}\rho^{1-n}\biggl(1-(n-1)\langle \Theta, \tau \rangle\biggl) +P_1(\rho,
\Theta,w,\tau)\biggl](-\Theta+\Upsilon),
\end{equation}
where 
$$\Upsilon=\sum^{n-1}_{m=1}a_k \Theta_k\quad \textrm{and}\quad a_k=\langle\nabla_{\Ss
}w,\Theta_k\rangle+P_0(\rho, \Theta,w,\tau).$$
\end{lem}

\begin{proof}
For each $\Theta \in \Ss$ the vector fields  
\begin{align}
\Theta_\ell:=&\de_{\ell}\Theta(s),\qquad \ell=1,\dots, n-1 
\end{align}
span the tangent space $T_\Theta \Ss$. 
Since  $\langle \Theta, \Theta \rangle = 1$, we have   that  $\langle \Theta,
\Theta_\ell \rangle =0 $ and $\la \n_{\Ss} w, \Theta\ra_{g_\Ss} =0$ on
$\Ss$. Without loss of generality, we may assume that $ \la
\Theta_i, \Theta_j\ra =\d_{ij}$.

 We look for a  normal vector $\widehat{\eta}$ of $\Ss$ with
respect to the metric $\widehat{g}$  in the form
\begin{equation}\label{eq: formnor}
\widehat{\eta}=-\Theta+\Upsilon, \qquad 
\Upsilon = \sum_{m=1}^{n-1} a_k \Theta_k .
 \end{equation} 
The condition that $\widehat \eta$ is normal 
is thus equivalent to 
\begin{equation}\label{eq:fiu}
\langle \widehat{\eta},\Theta_\ell \rangle_{g_{\rho}}=0,\quad \ell=1,...,n-1
\Leftrightarrow 
\langle \Upsilon, \Theta_\ell \rangle_{g_{\rho}}=\langle
\Theta,\Theta_\ell\rangle_{g_{\rho}},\quad \ell=1,...,n-1.
\end{equation}

By Lemma \ref{eq:metric} 
\begin{align}\label{eq: prod}
\langle \Theta ,\Theta_\ell\rangle_{g_{\rho}}&=
(1+2w+\sigma |z|^{1-n}) \langle \Theta ,\Theta_\ell\rangle\nonumber\\
&+ \left [ \Theta^i w_j\Theta^i \Theta_\ell^j
 + \Theta^jw_i\Theta^i \Theta_\ell^j +P_1(\rho, \Theta,w,\tau) \right ] \nonumber\\
&=\biggl[\langle\nabla_{\Ss }w,\Theta_\ell \rangle +P_1(\rho,
\Theta,w,\tau) \biggl]
\end{align}
and
\begin{align*}
\langle\Upsilon,\Theta_\ell\rangle _{g_{\rho}}&=\sum^{n-1}_{k=1}a_k\langle\Theta_k,\Theta_\ell\rangle_{g_{\rho}}=
\sum^{n-1}_{k=1}a_k\biggl((1+2w+\sigma |z|^{1-n})\delta_{k\ell}+P_1(\rho,\Theta,w,\tau)\biggl).
\end{align*}
Substituting these last two expressions into \eqref{eq:fiu}, we obtain 
\begin{equation}\label{eq:syste}
\sum^{n-1}_{k=1}a_k\widetilde{g}_{k\ell}=b_\ell, \quad \ell=1,...,n-1,
\end{equation}
 where
$$\widetilde{g}_{k\ell}:=(1+2w+\sigma|y|^{1-n})\delta_{k\ell}+P_1(\rho,\Theta,w,\tau)\quad
\textrm{and} \quad b_\ell:=\langle\nabla_{\Ss }w,\Theta_\ell\rangle +P_1(\rho,
\Theta,w,\tau), $$ 
so that \eqref{eq:syste} then implies 
\begin{equation}\label{eq: coefff}
a_k=\sum^{n-1}_{\ell=1}b_\ell\widetilde{g}^{k\ell}=\langle\nabla_{\Ss
}w,\Theta_k\rangle+P_1(\rho, \Theta,w,\tau).
\end{equation}

Next we compute $$\langle \widehat{\eta},\widehat{\eta}\rangle_{g_{\rho}}=\langle
\Theta, \Theta\rangle_{g_{\rho}}-2\langle
\Theta,\Upsilon\rangle_{g_{\rho}}+\langle
\Upsilon,\Upsilon\rangle_{g_{\rho}}.$$
We have $$\langle \Theta,
\Theta\rangle_{g_{\rho}}=\biggl(1+2w+\sigma|z|^{1-n}+P_0(\rho,
\Theta,w,\tau)\biggl)$$ Also, from \eqref{eq: coefff}, $\langle
\Upsilon,\Upsilon\rangle _{g_{\rho}}=P_1(\rho, \Theta,w,\tau)$. Using
once more \eqref{eq: prod} and \eqref{eq: coefff}, we obtain 
$$\langle
\Theta,\Upsilon\rangle_{g_{\rho}}=\sum^{n-1}_{k=1}a_k\langle
\Theta,\Theta_\ell \rangle_{\widehat{g}}=P_1(\rho, \Theta,w,\tau)$$
 and hence $$\langle
 \widehat{\eta},\widehat{\eta}\rangle_{g_{\rho}}=\biggl(1+2w+\sigma |z|^{1-n}+P_1(\rho,
\Theta,w,\tau)\biggl).$$ 

The normal interior unit vector field  to
$\mathbf{B}$ for the metric $g_{\rho}$ is the given by
\begin{equation}\label{eq:vectuni}
\widehat{\nu}_{g_{\rho}}=\frac{\widehat{\eta}}{\sqrt{  \langle \widehat{\eta},
\widehat{\eta}\rangle_{g_{\rho}}
}}=\biggl(1-w-\frac{\sigma}{2} |z|^{1-n}+P_1(\rho,
\Theta,w,\tau)\biggl)(-\Theta+\Upsilon) ,
\end{equation}
and so \eqref{eq:vectuni-in-lem} follows from \eqref{eqz}.\\
\end{proof}

\subsection{Expansion of the  normal derivative} 
\label{sec:expannormalderiv} 
The following proposition yields the expansion of the normal derivative 
of $\widehat{u}_{\rho,\tau, w}$ with respect to the metric $g_{\rho}.$
\begin{prop}\label{lem Norddd}
For $\rho$ sufficiently small the normal derivative of
$\widehat{u}=\widehat{u}_{\rho,\tau, w}= v+ \Psi_{\rho,\tau w}$ with respect to the
metric $g_{\rho}$ on $\de \B$ is given by
\begin{align}\label{eq: der1v2}
\frac{\partial \widehat{u}}{\partial \widehat{\nu}_{g_{\rho}}}=(n-2)(1-\frac{\sigma}{2}\rho^{n-1})+ H(\rho, \tau, w),
\end{align}
where 
\begin{equation}\label{equ:H}
H(\rho, \tau, w)(x):=(n-1)(n-2)\biggl(\frac{\sigma}{2}\rho^{n-1}\langle x, \tau \rangle-w+P_1(1/\rho,
x,w,\tau)\biggl)+\widehat{\nu}^i_{g_{\rho}}\frac{\partial
\Psi_{\rho,\tau, w}}{\partial x_i}
\end{equation}
and  $\Psi_{\rho,\tau, w}$ is solution of  \eqref{eq:extremal-capacitor2r}.
\end{prop}

\begin{proof}
Recall our solution $ \widehat{u}_{\rho,\tau, w}(x)=v (x)+\Psi_{\rho,\tau, w}(x)$, where 
$v$ is given by \eqref{eq:vv} and $\Psi_{\rho,\tau,w}$ satisfies \eqref{eq:extremal-capacitor2r}. 

We  first compute $g_{\rho}(\nabla_{g_{\rho}}v,\widehat{\nu}_{g_{\rho}})$.
By Lemma \ref{eq:metric} and \eqref{eq:1plaplv} 
\begin{align*}\label{derngardo}
\nabla_{g_{\rho}}^jv&=\sum^{n}_{i=1}g_{\rho}^{ij}\frac{\partial
v}{\partial x_i}=\biggl[ (1-2w-\sigma |z|^{1-n})
\delta_{ij} - ( x_i w_j + x_j
w_i)+ P_0(\rho,x,w,\tau)\biggl]\times\\
&\biggl[-(n-2)\biggl(1-(n-2)w+Q^{0}(w)\biggl)x_i-(n-2)w_i+Q^{1}(w)\biggl]\\
&= -(n-2)\biggl((1-\sigma |z|^{1-n})x_j-nw
x_j+ P_1(\rho,x,w,\tau)\biggl)\quad on\quad  \partial\mathbf{B}, 
 \end{align*}
 which yields 
 $$g_{\rho}(-\Theta, \nabla_{g_{\rho}}v )=- \biggl[ (1+2w+\sigma |z|^{1-n}) \langle\Theta,\nabla_{g_{\rho}}v\rangle + \Theta^i
w_j\Theta^i\nabla^{j}_{g_{\rho}}v + \Theta^j
w_i\Theta^i\nabla^{j}_{g_{\rho}}v +P_0(\rho,\Theta,w,\tau)\biggl].$$
On the other hand, 
\begin{align*}
\langle\Theta,\nabla_{g_{\rho}}v\rangle&=-(n-2)\biggl(1-\sigma |z|^{1-n}-nw+P_1(\rho,\Theta,w,\tau)\biggl), 
\end{align*}
so we may rewrite our expression for $g_\rho (-\Theta, \nabla_{g_\rho} v )$ as 
\begin{equation}\label{eq:dernor1}
g_{\rho}(-\Theta, \nabla_{g_{\rho}}v
)=(n-2)\biggl(1+(2-n)w+P_1(\rho,\Theta,w,\tau)\biggl).
\end{equation}
Additionally, 
\begin{equation}\label{eq:dernor2}
g_{\rho}(\Upsilon, \nabla_{g_{\rho}}v
)=\sum^n_{k=1}a_kg_{\rho}(\Theta_k, \nabla_{g_{\rho}}v
)=\sum^n_{k=1}a_kg_{\rho}(\Theta_k, \Theta)+ P_1(\rho,\Theta,w,\tau)= P_1(\rho,\Theta,w,\tau),
\end{equation}
where we have used \eqref{eq: prod} in the last equality.
Finally, making use of  \eqref{eq:vectuni}, \eqref{eq:dernor2} and  \eqref{eq:dernor1}, we obtain 
\begin{align}\label{eq: derv1}
g_{\rho}(\nabla_{g_{\rho}}v,\widehat{\nu}_{g_{\rho}})&=
\biggl(1-w-\frac{\sigma}{2}|z|^{1-n}+P_1(\rho,
\Theta,w,\tau) \biggl)g_{\rho}(-\Theta+\Upsilon, \nabla_{\widehat{g}}v )\nonumber\\
 &=(n-2)\biggl(1-w-\frac{\sigma}{2}|z|^{1-n}+P_1(\rho,
\Theta,w,\tau) \biggl)\biggl(1+(2-n)w+P_1(\rho,\Theta,w)\biggl)\nonumber\\
&=(n-2)\biggl(1-\frac{\sigma}{2}\rho^{1-n}+\frac{\sigma(n-1)}{2}\rho^{1-n}\langle x, \tau \rangle-(n-1)w+P_1(\rho,
\Theta,w,\tau)\biggl)
\end{align}
and 
\begin{align}\label{eq: derv3}
g_{\rho}(\nabla_{g_{\rho}} \Psi_{\rho,\tau, w},\widehat{\nu}_{g_{\rho}})&=g_{\rho}^{ij}(g_{\rho})_{jk}\widehat{\nu}^k_{g_{\rho}}\frac{\partial
\Psi_{\rho,\tau, w}}{\partial x_i}=\delta_{ik}\widehat{\nu}^k_{g_{\rho}}\frac{\partial
\Psi_{\rho,\tau, w}}{\partial x_i}=\widehat{\nu}^i_{g_{\rho}}\frac{\partial
\Psi_{\rho,\tau, w}}{\partial x_i}.
\end{align}
The expansion \eqref{eq: der1v2}  then follows from  \eqref{eq: derv1} and \eqref{eq: derv3} replacing $\rho$ by $1/\rho$. \\
\end{proof}

\subsection{Linear analysis of the normal derivative} 
\label{sec:linearana} 

The leading term in \eqref{eq:   der1v2} is $$(n-2) (1-\sigma \rho^{n-1} /2),$$ and to complete our 
proof we must throughly understand the reminder term $H(\rho,\tau,w)$. To assist in this analysis 
we linearize the function $G(\rho,\tau,w) = \rho^{1-n} H(\rho,\tau \rho^{n-1} w)$. 
\begin{prop}\label{eq:diffG}
Defining   
\begin{equation}\label{eq:defG}
G(\rho, \tau, w):=\rho^{1-n}H(\rho, \tau, \rho^{n-1} w),
\end{equation}
 we have
\begin{equation}\label{eq:dertauG}
[D_\tau \big|_{(\rho,\tau, w)=(0,0,0)} G(\rho,\tau, w)] \cdot \tau =\dfrac{(n-1)(n-2)}{2} \biggl(1+\dfrac{(n-2)(n-4)}{2}\biggl)\sigma \langle x, \tau \rangle
\end{equation}
and 
\begin{equation}\label{eq:derw}
\mathbb{L}(w):=[D_w \big|_{(\rho,\tau, w)=(0,0,0)} G(\rho,\tau, w)] \cdot w=\partial_{\nu} \psi_w\cdot x -(n-1)  w,
\end{equation}
where  $\nu=-x$ and  $\psi_{w}$ is the
unique solution to
\begin{equation}\label{eq:extremal-capcff1}
\begin{cases}
\Delta_0 \psi_{w} &=0\quad \textrm { in } \quad\R^n
\backslash \B \vspace{4mm}\\
\psi_{w}&= w \quad \textrm{ on }\quad
\partial \B.\\
% \vspace{3mm}
\end{cases}
\end{equation}
\end{prop}

\begin{proof}
From  \eqref{eq:defG} and \eqref{equ:H}, we have 
\begin{align}\label{equ:solve2}
G(\rho, \tau, w)(x)&=(n-1)(n-2)\biggl(\frac{\sigma}{2} \langle x, \tau \rangle-w+\rho^{1-n}P_1(1/\rho,
\Theta,\rho^{n-1}w,\tau)\biggl)\nonumber\\
&+\widehat{\nu}^i_{g_{\rho}}(\rho^{n-1}w)\frac{\partial
\overline{\Psi}_{\rho,\tau, w}}{\partial x_i},
\end{align}
where   $$ \overline{\Psi}_{\rho,\tau, w}:= \rho^{1-n} \Psi_{\rho,\tau, \rho^{n-1} w},$$ and 
$\widehat{\nu}^i_{g_{\rho}}(\rho^{n-1}w)$ is the ith-component of the unit vector $\widehat{\nu}_{g_{\rho}}$ in  \eqref{eq: der1v2} with $\rho^{n-1}w$ in place of  $w$. 
We see from  \eqref{eq:extremal-capacitor2r} and Lemma \ref{eq:lapv} that  $\overline{\Psi}_{\rho,\tau, w}$ is solution the unique solution of 

\begin{equation}\label{eq:extremal-capacitor222}
\begin{cases}
\Delta_{g_{\rho}} \overline{\Psi}_{\rho,\tau, w} &= -\dfrac{(n-1)(n-2)^2}{2} \sigma r^{1-2n}+\dfrac{(n-1)(n-4)(n-2)^2}{2} \sigma r^{1-2n}\langle  \frac{x}{r^2}, \tau \rangle\\
& +  \rho^{1-n}P_2(1/ \rho,x, \rho^{n-1}w,\tau)  \quad
\textrm { in } \quad\R^n
\backslash \B \vspace{4mm}\\
\overline{\Psi}_{\rho,\tau, w}&=(n-2)w +\rho^{1-n}Q^0(\rho^{n-1}w,\tau)  \quad \textrm{ on }\quad
\partial \B.
% \vspace{3mm}
\end{cases}
\end{equation} 
Differentiating  \eqref{eq:extremal-capacitor222} with respect to $\tau$ at $(\rho,\tau, w)=(0,0,0)$, we see 
$$F_{\t}:=[D_\tau \big|_{(\rho,\tau, w)=(0,0,0)} \overline{\Psi}_{\rho,\tau, w}] \cdot \tau$$  satisfies
\begin{align}\label{eq:def-Psi-tau}
\begin{cases}
\Delta  F_{\t} =C(n,\sigma) r^{1-2n}\langle  \frac{x}{r^2}, \tau \rangle&
\textrm { in } \quad\R^n
\backslash \B \vspace{3mm}\\
   F_{\t}  =0&\quad
\textrm{ on }  \partial \B,
% \vspace{3mm}
\end{cases}
\end{align}
where  
\begin{equation}\label{eq:cn}
C(n,\sigma):=\dfrac{(n-1)(n-4)(n-2)^2}{2} \sigma.
\end{equation}
We observe from \eqref{eq:vectuni-in-lem} that 
\begin{align}\label{eq:def-tow}
[D_w\big|_{(\rho,\tau, w)=(0,0,0)} \widehat{\nu}^i_{g_{\rho}}(\rho^{n-1}w)]\cdot w=0=[D_\tau \big|_{(\rho,\tau, w)=(0,0,0)} \widehat{\nu}^i_{g_{\rho}}(\rho^{n-1}w)]\cdot \tau,
\end{align}
 which  combined with  \eqref{equ:H} and \eqref{eq:defG} allow to get 
 \begin{equation}\label{eq:dernortau}
[D_\tau \big|_{(\rho,\tau, w)=(0,0,0)} G(\rho,\tau, w)] \cdot \tau =\dfrac{(n-1)(n-2)}{2} \sigma \langle x, \tau \rangle-x \cdot \n F_{\t}.
\end{equation}

Next, we show that  
\begin{equation}\label{eq:cdertau}
\de_{\nu } F_{\t}=    \frac{C(n,\sigma)}{2(n-1)}\la \cdot, \t\ra
\qquad \textrm{ on }   \Ss, \quad \textrm{with}  \quad \nu=-x.
\end{equation}
Indeed, let 
$$
\Pi: L^2( \Ss ) \to L^2( \Ss )
$$
be the orthogonal projection on $\textrm{span}\{x_1;\dots ;x_n\}$
and consider  the function
$$
X^i(x):=\mathbb{K}(x^i)=|x|^{2-n}\frac{x^i}{|x|^2}.
$$
We know that $\D_0 X^i=0$ in $\R^n\setminus\{0\}$. We multiply
\eqref{eq:def-Psi-tau} with $X^i$ and integrate by parts to get
\begin{align}\label{eq:dernorss}
\int_{\Ss}\de_{\nu}  F_{\t}(x) x^i \, d\mu_S&=C(n,\sigma)\int_{\R^n
\backslash\B} |x|^{-1-n}  |x|^{2-n}\la\frac{x}{|x|^2},\t \ra  X^i(x)
\, d\mu \nonumber\\
&=C(n,\sigma)\t^i \int_{\Ss}(x^i)^2d\mu\int_1^\infty r^{1-3n} r^{n-1}\,
dr\nonumber.
\end{align}
This implies that \be \int_{\Ss}\de_{\nu}  F_{\t}(x) x^i \,
d\mu_S=\,\frac{C(n,\sigma)|\partial\B|}{2n-1}\t^i=  \frac{C(n,\sigma)}{2n-1} \int_{\Ss} \la
x,\t\ra x^i  d\mu_S. \ee
From this, we deduce that \be \label{eq:Project-Psit}
 \Pi \de_{\nu } F_{\t}(x)=    \frac{C(n,\sigma)\la \t,x\ra}{2n-1} \qquad \textrm{ for  } x\in  \Ss.
\ee 

We \textbf{claimed} that \be \label{eq:Project-Psit-perp}
 \Pi^\perp  \de_{\nu }F_{\t} = 0. 
\ee
To see this, we let $Y^k_i$ be the spherical harmonics for which
$Y^1_i=x^i$ for $i=1,\dots,n$, corresponding to the eigenvalues
$k(k+n-2)$ on sphere.  We suppose that $k\neq 1$. We then define
$$
X^k_i(x)= \mathbb{K}( Y^k_i)(x).
$$
Then $X^k_i$ are admissible test functions in
\eqref{eq:def-Psi-tau}. We observe that the right hand side in
\eqref{eq:def-Psi-tau} is in $L^2(\R^n\setminus \B)$. Therefore by
simple arguments, we have that  $F_\t\in H^1(\R^n\setminus \B )
$. Using the decomposition in spherical harmonics of $ F_\t$, we
can see that $F_\t(x)=f(|x|) \la x, \t\ra$, for some some
function $f$. From this
 we can multiply \eqref{eq:def-Psi-tau} by $X^k_i$ and use the Gauss-Green formula 
 to  deduce that
$$
\int_{\Ss}\de_{\nu}  F_{\t}(x) Y^k_i(x) \, d\mu_S=\t^i \int_{\R^n
\backslash\B} |x|^{-n}  |x|^{-n}\  x^i  X^k_i(x) \, d\mu=0,\qquad
k\neq 1,
$$
as claimed.
Gathering \eqref{eq:cn}, \eqref{eq:dernortau}  and \eqref{eq:cdertau}, we obtain \eqref{eq:dertauG}.   

To complete the proof of Proposition \ref{eq:diffG}, we diffferentiate   \eqref{eq:extremal-capacitor222} with respect to $w$ at $(\rho,\tau, w)=(0,0,0)$  to get a function 
\begin{align}\label{eq:defww}
\psi_{w}:=\frac{1}{n-2} [D_w \big|_{(\rho,\tau, w)=(0,0,0)} \overline{\Psi}_{\rho,\tau, w}] \cdot w 
\end{align}
which solve uniquely
\begin{align}\label{eq:def-w1}
\begin{cases}
\Delta  \psi_{w} = 0&
\textrm { in } \quad\R^n
\backslash \B \vspace{3mm}\\
   \psi_{w}  =w&\quad
\textrm{ on }  \partial \B.
% \vspace{3mm}
\end{cases}
\end{align}
The equality in \eqref{eq:derw} then follows from  \eqref{equ:solve2}, \eqref{eq:defww}  and  \eqref{eq:def-tow}. 
\end {proof}

\subsection{The spectral properties  of the operator $\mathbb{L}$}\label{secapecta}
We  study the spectral properties of the operator $\mathbb{L}$  in \eqref{eq:derw} defined by 
$$\mathbb{L}(w):=\partial_{\nu}\psi_{w}-(n-1)w,$$
 where  $\psi_{w}$ satisfies \eqref{eq:extremal-capcff1}.
We consider the Kelvin transform (see \eqref{defn-kelvin-trans} and
\eqref{kelvin-trans-law}) of $\psi_w$
$$
\mathbb{K}(\psi_w) (x) = |x|^{2-n}
\psi_w\left(\frac{x}{|x|^2}\right),
$$
which satisfies
$$
\Delta  (\mathbb{K}(\psi_w) ) =0\qquad \textrm{ in }
\B\setminus\{0\}.
$$
Moreover by direct computations, we have
\begin{equation}\label{eq:dd}
 \nabla (\mathbb{K} (\psi_w))\cdot x= -\nabla \psi_w\cdot x   +
(2-n)  w=\partial_{\nu}\psi_{w}+ (2-n)w   \qquad \textrm{ in }
\partial\B .
\end{equation}
Since $\psi_w\in\mathcal{C}^{2,\alpha}_\nu(\R^n \backslash \B)$ with
$\nu\in(2-n,0)$, we see that
$$
\lim_{|x|\to \infty}  |x|^{2-n}\psi_w(x)=0.
$$
Therefore
$$
\lim_{|x|\to 0}  |x|^{n-2}\mathbb{K}(\psi_w)(x)=0,
$$
so the origin is a removable singularity  for  $\mathbb{K} (\psi_w)$ and  thus
$\mathbb{K}(\psi_w)$ solves
\begin{equation}\label{eq:extre-capa5}
\begin{cases}
\Delta \mathbb{K}( \psi_w)=0  \quad \textrm { in } \quad\B
\vspace{3mm}\\
   \mathbb{K}(\psi_w)=w \quad
\textrm{ on }\quad \partial \B. \vspace{3mm}
\end{cases}
\end{equation}
By elliptic regularity theory, $\mathbb{K}(\psi_w)\in
\mathcal{C}^{2,\alpha}(\overline{\B}) $. From \eqref{eq:dd} and the
definition of $\mathbb{L}$ we get
$$
  w\mapsto \mathbb{L}(w)=\partial_{\tilde{\nu}}(\mathbb{K}(\psi_w))-w,
$$
where $\tilde{\nu}=\Theta$.

Thank to \cite{RaSa}, the spectrum of the operator $\mathbb{L}$ is
given by
\begin{equation}\label{eq; spec}
\lambda_k=k-1,\quad k\in \mathbb{N}
\end{equation}
meaning that the kernel of the operator $\mathbb{L}$ is given by the
space $V_1$  spanned by linear coordinates on the sphere
$\mathbf{S}$
$$V_1:=\bigl\{\Theta^i, i=1,...,n\bigl\}.$$
%\end{proof}
Moreover there exists a constant $C>0$ such that 
\begin{equation}\label{eq:inverl}
||w||_{C^{2,\alpha}} \leqslant C ||\mathbb{L}(w)||_{C^{1,\alpha}(\mathbf{S})},
\end{equation}
provided $w \in \Pi^{\perp} C^{,\alpha}(\mathbf{S}).$

\section{Solving the nonlinear problem} \label{sec:nonlinear} 

Define the mapping 
\begin{equation}\label{eq:nolinear2}
\widetilde{G}(\rho, \tau, w):=   G(\rho, \tau, w)+\dfrac{(n-1)(n-2)^2}{2} \sigma\partial_{\nu} \mathcal{K} =G(\rho, \tau, w)-\dfrac{(n-1)(n-2)^2}{2} \sigma x \cdot \n \mathcal{K} ,
\end{equation}
where $\mathcal{K} $ is the unique solution of 
\begin{equation}\label{eq:epsi}
\begin{cases}
\Delta \mathcal{K} &=  r^{1-2n} \quad
\textrm { in } \quad\R^n
\backslash \B \vspace{4mm}\\
\mathcal{K} &=0  \quad \textrm{ on }\quad
\partial \B.
% \vspace{3mm}
\end{cases}
\end{equation}
Then from  \eqref{eq:nolinear2} and   \eqref{eq:extremal-capacitor222},
\begin{equation}\label{eq:beinimplici}
\widetilde{G}(0,0,0)=0. 
\end{equation}
We want to prove that provided $\rho$ is small, we can  find $\tau$ and $w$ such that 
\begin{equation}\label{eq:sonenonline}
\widetilde{G}(\rho, \tau, w)=0. 
\end{equation}

We denote by  $\Pi$ the  orthogonal projection from $\mathcal{C}^{1,\alpha}(\Ss)$ 
onto $V_1$ and  $T:V_1 \rightarrow \R^n$ the  isomorphisim sending $x_i\big|_{\partial \B}$ 
to $e_i$. We also define  $\widetilde{\Pi}:=T\circ \Pi$,
$$\widetilde{K}:=\widetilde{\Pi}\circ \widetilde{G}:   (0,\rho_0)\times
\B_{c_0}(0)\times \B_{c_0}(0)   \longrightarrow \R^n $$ and consider  the equation 
\begin{equation}\label{eq:nonlietosolv}
\widetilde{K}(\rho, \tau, w)=0. 
\end{equation} The mapping  $\widetilde{K}$ has the following properties:
\begin{itemize}
\item
$\widetilde{K}(0, 0, 0)=0$. This is from \eqref{eq:beinimplici},
\item
$D_\tau \big|_{(\rho,\tau, w)=(0,0,0)} \widetilde{K}$
is a  the identity in $\R^n$ times a constant, which  follows from \eqref{eq:dertauG}. 
\end{itemize}
Applying the implicit function theorem, 
we find a unique smooth mapping  $$(-r_0,r_0)\times\B_{k_1}(0)\longrightarrow \B_{k_2}(0) \subset \R^n, \quad (\rho, w)\mapsto \tau(\rho, w)$$  defined for some positive constants $r_0$, $k_1$  and  $k_2$ such that 
\begin{equation}\label{eq:impli1}
\widetilde{K}(\rho, \tau(\rho,w), w)=0 \quad\textrm{ for all} \quad  (\rho, w)\in  (-r_0,r_0)\times\B_{k_0}(0).
\end{equation}
Net we provide estimates for the function $\tau(\rho,w)$ in \eqref{eq:impli1}.
Observe that since   $ \widetilde{\Pi}\circ  \mathbb{L}=0,$  \eqref{eq:nonlietosolv} is equivalent to 
\begin{equation}\label{eq:proj}
C'(n,\sigma) \tau+\widetilde{\Pi} [\rho^{1-n}P_1(1/ \rho,x, \rho^{n-1}w,\tau)]=0\quad \textrm{on}\quad \mathbf{S},
\end{equation}
where $C'(n,\sigma)$ is the constant appearing in \eqref{eq:dertauG}. This can be seen by writing the Taylor expansion of $G$ using  \eqref{eq:dertauG} and  \eqref{eq:derw}. We then deduce  from \eqref{eq:estiL}  and  \eqref{eq:def-of-P-i} the estimates
\begin{equation}\label{eq:derivestau}
|\tau(\rho, w)|\leq C\rho\quad \textrm{and}\quad |D_w \tau(\rho, w)|\leq C\rho^{n-1}.
\end{equation}
Now replace $\tau$ by $\tau(\rho, w)$ in  \eqref{eq:nolinear2} and consider the equation 
\begin{equation}\label{eq:nonliew}
\Pi^{\perp}(\widetilde{G}(\rho, \tau, w))=0. 
\end{equation}
From \eqref{eq:derw} and the estimates in \eqref{eq:derivestau},
\begin{equation}\label{eq:derw2}
D_w \big|_{(\rho,\tau, w)=(0,0,0)} \Pi^{\perp}\circ \widetilde{G}=\mathbb{L}: V_1^{\perp}\longrightarrow \mathbb{L}(V_1^{\perp}),
\end{equation}
which is an isomorphisim from Subsection \ref{secapecta}. Hence, there exists a unique solution $w(\rho)$ to \eqref{eq:nonliew}  for small  $\rho\in (0,R_0)$.\\
Using \eqref{eq:dertauG} and  \eqref{eq:derw}, we have
\begin{equation}\label{eq:proj}
\mathbb{L}(w_\rho)+\Pi^{\perp}[\rho^{1-n}P_1(1/ \rho,x, \rho^{n-1}w,\tau)]=0\quad \textrm{on}\quad \mathbf{S},
\end{equation}
and  from  \eqref{eq:inverl} and \eqref{eq:def-of-P-i},
\begin{equation}\label{eq:estiw}
||w_\rho||_{C^{2,\alpha}(\mathbf{S})} \leq C \rho.  
\end{equation}

Decreasing $R_0$ if necessary, the  analysis of the previous section establishes the first
statement of Theorem \ref{main-thm} with   $\rho_0=\frac{1}{R_0}$ and $$\Omega_{\rho}=\widetilde{\Omega_{\frac{1}{\rho}}},\quad \rho\in (\rho_0,+\infty),$$
where 
\begin{equation}\label{eq:domain}
\widetilde{\Omega_{\rho}}:= \Phi_{\frac{1}{\rho}, \tau(\rho, w(\rho)), \rho^{n-1}w(\rho)}(\R^n\backslash \mathbf{B}),\quad \rho \in (0, R_0).
\end{equation}
In addition, recalling \eqref{eq:chanageof metric}, we have $\widehat \nu_{\widehat{g}}=\rho \widehat \nu_{g_{\rho}}$ for small $\rho$ and from \eqref{eq: der1v2}  and \eqref{eq:nolinear2}, we see that  the constant $C(\rho, \sigma, n)$ in  Theorem \ref{main-thm}   is given by  

\begin{equation}\label{eqNeumanncostant}
C(\rho, \sigma, n)=\frac{n-2}{\rho}(1-\frac{\sigma}{2}\rho^{1-n})-\dfrac{(n-1)(n-2)^2}{2} \sigma\rho^{-n}x\cdot \n \mathcal{K}(x)\big|_{\mathbf{S}},
\end{equation}
where  $\mathcal{K}$ is the unique solution of \eqref{eq:epsi}.
It is plain that the function  $\mathcal{K}$ is radial. By  Gauss-Green formula, we get
$\mathcal{K}'(1)=-\frac{1}{n-1}$ and we deduce
\begin{equation}\label{eqNeumanncost}
C(\rho, \sigma, n)=\frac{n-2}{\rho}+\dfrac{(n-2)(n-3)}{2\rho^{n}} \sigma.
\end{equation}

It remains to show that the
family $(\partial \O_\rho)_{\rho\in(\rho_0, +\infty)}$ constitutes a smooth foliation.

\section{Foliation by boundaries of extremal capacitors}
\label{sec:foliation} 

\begin{prop}\label{eq:folia2}
There exists a constant $\rho_0>1$ such that the family $(\partial
K_\rho)_{\rho>\rho_0}$ constitutes a smooth foliation.
\end{prop}
\begin{proof}

We are proving that the family $(\partial \O_\rho)_{\rho \in(\rho_0, +\infty)}$  constitutes a foliation of $M\backslash \O_{\rho_0}$. 
The proof is inspired by the argument in  \cite[Section 5]{FM} and  \cite[Pages 9-10]{RY}.  

Notice that $\partial \widetilde{\O_\rho}$ is given by $$\partial \widetilde{\O_\rho}=\Phi(\widetilde{S}_{\frac{1}{\rho}, \tau(\rho, w(\rho)), \rho^{n-1}w(\rho)}),$$
where 
\begin{equation}\label{eq:folia1}
\widetilde{S}_{\frac{1}{\rho}, \tau, w}=\biggl\{y=\frac{1}{\rho}\biggl(x + \tau +w(x)x\biggl),\quad
x\in \Ss \biggl\}.
\end{equation}
We define the functions
$$h(\rho,x):=\frac{1}{\rho}\biggl(x + \tau(\rho, w(\rho))+\rho^{n-1}w_\rho(x) x\biggl)\quad\textrm{and}\quad
v(\rho,x):=\frac{h(\rho,x)}{|h(\rho,x)|_{g_{\rho}}},
\quad x\in \Ss.$$ 

By the estimates in \eqref{eq:derivestau} and \eqref{eq:estiw},  the function $v(\rho,\cdot)$ extends smoothly at
$\rho=0$ with $v(0,\cdot)=I_{\Ss}$ and for all $\rho$
small $v(\rho,\cdot)$  is a diffeomorphism from $\Ss$ into
itself.

Thus for all $y\in \Ss$, $$h(\rho,v^{-1}(\rho,
y))=|h(\rho,v^{-1}(\rho,
y))|_{g_{\rho}}y.$$  We put
$$\tilde{\varphi}(\rho,y):=|h(\rho,v^{-1}(\rho,
y))|_{g_{\rho}},$$ where $v^{-1}(\rho, \cdot)$
denotes the inverse of $v(\rho, \cdot)$.  Then 
\begin{equation}\label{eq:folia2}
\widetilde{S}_{\frac{1}{\rho}, \tau(\rho, w(\rho)), \rho^{n-1}w(\rho)}=\widetilde{S}_{\rho}:=\biggl\{\tilde{\varphi}(\rho,y)y,
\quad y\in \Ss\biggl\}.
 \end{equation}
 
Using  the estimates in \eqref{eq:derivestau} and \eqref{eq:estiw} once again, we find 
\begin{align*}
D_{x}h&=\frac{1}{\rho}\biggl(I_{\Ss}+\rho^{n-1}L^1(w_\rho)x+\rho^{n-1}w_\rho(x)I_{\Ss}\biggl)=\frac{1}{\rho}(I_{\Ss}+O(\rho)) \\
\frac{\partial h}{\partial
\rho}&=-\frac{1}{\rho^2}\biggl(x+\tau -\rho \frac{\partial \tau}{\partial \rho}-\rho D_w\tau(\rho,w(\rho))w'_\rho+\rho^{n} w'_\rho(x) x-(n-2)\rho^{n-1}w_\rho(x)x\biggl)\nonumber\\
&=-\frac{1}{\rho^2}(x+O(\rho) ).
\end{align*}
Also, since $$|v^{-1}(\rho, y)|_{g_{\rho}}=1\quad\textrm{for
all}\quad\rho\in(0,R_0)\quad\textrm{and}\quad y\in
\Ss,$$ we have $$\langle v^{-1}(\rho,
y),\partial_{\rho}v^{-1}(\rho,
y)\rangle_{{g}_{\rho}}=0\quad\textrm{for
all}\quad\rho \in(0,R_0)\quad\textrm{and}\quad y\in
\Ss.$$

 Using this, we  then obtain
 \begin{align*}\frac{\partial
\tilde{\varphi}}{\partial \rho
}&=\frac{1}{|h(\rho,v^{-1}(\rho,
y))|_{{g}_{\rho}}}\langle
h(\rho,v^{-1}(\rho, y)),\frac{\partial h}{\partial
\rho}(\rho,v^{-1}(\rho, y))+D_xh(\partial_{\rho}
v^{-1})\rangle_{{g}_{\rho}}\\
&=\frac{1}{|v^{-1}(\rho,
y)+O(\rho)|_{{g}_{\rho}}}\langle v^{-1}(\rho,
y)+O(\rho),-\frac{1}{\rho}v^{-1}(\rho,
y)+\partial_{\rho}
v^{-1}+O(\rho)\rangle_{{g}_{\rho}}\\
&=\frac{1}{\rho|v^{-1}(\rho,
y)+O(\rho)|_{{g}_{\rho}}}\langle v^{-1}(\rho,
y)+O(\rho),-v^{-1}(\rho,
y)+\rho\partial_{\rho}
v^{-1}+O(\rho^2)\rangle_{{g}_{\rho}}\\
&=-\frac{1}{\rho}(1+O(\rho)).
\end{align*}
We conclude the the function $\tilde{\varphi}(\rho, x)$ is
strictly decreasing with respect to $\rho$ for $\rho$
small or equivalently $\tilde{\varphi}(\rho^{-1}, x)$ is strictly
increasing for $\rho$ large.  Thank to \eqref{eq:folia2},  the family
$(\widetilde{S}_{\frac{1}{\rho}})_{\rho>\rho_0}$ constitutes
a foliation of $\R^n\backslash \mathbf{B}_{\rho_0}$.  Since $\Phi$ is a diffeomorphism and $\partial \O_\rho=\Phi(\widetilde{S}_{\frac{1}{\rho}})$, we deduce that the family $(\partial \O_\rho)_{\rho\in(\rho_0, +\infty)}$ foliates $M\backslash \O_{\rho_0}$ and the proof of Theorem  \ref{main-thm} is complete. 
\end{proof}

\appendix 
\section{Variational setting}
\label{sec:euler-lagrange}

In this section we define two scale-invariant energies associated with 
capacity, and compute their first variations. We perform 
the computations below in Euclidean space, but again it is 
an easy exercise to carry them out in a Riemannian manifold with 
one asymptotically flat end. 

We earlier 
defined the capacity function $\operatorname{Cap}$  on compact 
sets $K \subset \R^n$ in \eqref{cap-defn}. A change of variables 
shows us the scaling law 
\begin {equation} \label{cap-scaling}
\operatorname{Cap}(RK) = R^{n-2} \operatorname{Cap}(K) \end{equation} 
for any $R>0$, and a quick compation gives $\operatorname
{Cap}(\overline{\mathbf{B}}_R) = R^{n-2}$, with the equilibrium potention 
function $U(x) = R^{n-2} |x|^{2-n}$. As we discussed in the introduction, 
one can either normalize using volume or surface area, leading to the 
following two scale-invariant functionals:
\begin {equation} \label{energy} 
\mathcal{E}_0(K) = \frac{\operatorname{Cap}(K)}{|K|^{\frac{n-2}{n}}}, \qquad 
\mathcal{E}_1(K) = \frac{\operatorname{Cap}(K)}{|\Sigma|^{\frac{n-2}{n-1}}}.
\end {equation}
By \eqref{cap-scaling} both $\mathcal{E}_0$ and $\mathcal{E}_1$ are scale-invariant. 

To compute the first variation of both $\mathcal{E}_0$ and $\mathcal{E}_1$ 
we let $X:(-\epsilon, \epsilon) \times \R^n \rightarrow \R^n$ be a vector field, 
and let $\xi$ be its flow, defined by 
$$\xi:(-\epsilon, \epsilon) \times \R^n \rightarrow \R^n, \qquad 
\frac{\partial \xi}{\partial t} (t,x) = X(t,x), \qquad \xi(0,x) = x. $$
Let $K_t = \xi(t,\cdot)(K)$ and let $U_t$ be the solution to \eqref{cap-pde} 
in $\Omega_t = \R^n \backslash K_t$. It will also be convenient to denote 
$\Sigma = \partial K = \partial \Omega$, and $\Sigma_t = \partial K_t = 
\partial \Omega_t$. 

\begin {lem} 
We have 
\begin {equation} \label{cap-EL} 
\left. \frac{d}{dt} \operatorname{Cap}(K_t) \right |_{t=0} = \frac{1}{n(n-2)\omega_n 
|K|^{\frac{n-2}{n}}} \left [ \frac{(n-2)^2 \omega_n \operatorname{Cap}(K)}{|K|} 
\int_\Sigma \langle X, \eta d\sigma - \int_\Sigma \langle X, \eta \rangle 
\left ( \frac{\partial U}{\partial \eta} \right )^2 d\sigma \right ], \end {equation} 
where $\eta$ is the unit normal vector pointing into $K$. 
\end {lem} 

\begin {proof} 
Observe that $\left. U_t \right |_{\Sigma_t} = 0$. Differentiating this boundary condition, 
we see 
$$ \frac{\partial U}{\partial t} = - \langle X, \eta \rangle \frac{\partial U}{\partial \eta} 
\textrm{ on } \Sigma_t.$$
\begin {eqnarray*} 
\left. \frac{d}{dt} \right |_{t=0} \mathcal{E}_0(K_t) & = & \frac{1}{n(n-2)\omega_n} 
\left. \frac{d}{dt} \right |_{t=0} |K|^{\frac{2-n}{n}} \int_{\Omega_t} |\nabla U_t|^2 dx \\ 
& = & \frac{(n-2)}{n} \frac{1}{n(n-2)\omega_n |K|^{\frac{n-2}{n} -1} } \int_\Sigma \langle X, \eta
\rangle d\sigma \\ 
& & + \frac{1}{n(n-2)\omega_n|K|^{\frac{n-2}{n}}} \left [ 2 \int_\Omega \left \langle 
\nabla U, \nabla \frac{\partial U}{\partial t} \right \rangle dx + \int_\Sigma 
\langle X, \eta \rangle |\nabla U|^2 d\sigma \right ] \\ 
& = &  \frac{(n-2)}{n} \frac{1}{n(n-2)\omega_n |K|^{\frac{n-2}{n} -1} } \int_\Sigma \langle X, \eta
\rangle d\sigma \\ 
&& + \frac{1}{n(n-2)\omega_n|K|^{\frac{n-2}{n}}} \left [ \int_\Sigma \langle X, \eta 
\rangle d\sigma - 2 \int_\Omega \frac{\partial U}{\partial t} \Delta U dx + 2 
\int_\Sigma \frac{\partial U}{\partial t} \frac{\partial U}{\partial \eta} \right ] \\ 
& = &  \frac{(n-2)}{n} \frac{1}{n(n-2)\omega_n |K|^{\frac{n-2}{n} -1} } \int_\Sigma \langle X, \eta
\rangle d\sigma \\ 
&&  \frac{1}{n(n-2)\omega_n|K|^{\frac{n-2}{n}}} \left [ \int_\Sigma \langle X, \eta 
\rangle d\sigma \right ] \\ 
& = & \frac{1}{n(n-2)\omega_n 
|K|^{\frac{n-2}{n}}} \left [ \frac{(n-2)^2 \omega_n \operatorname{Cap}(K)}{|K|} 
\int_\Sigma \langle X, \eta d\sigma - \int_\Sigma \langle X, \eta \rangle 
\left ( \frac{\partial U}{\partial \eta} \right )^2 d\sigma \right ]. 
\end {eqnarray*} 
Here we have used the fact that $U$ is constant on $\Sigma$, so $|\nabla U| = 
\frac{\partial U}{\partial \eta}$ there. 
\end {proof}

 \begin {cor} \label{over-determined-cor}
 A compact set $K$ with nonempty interior $\mathcal{O}$ and 
 smooth boundary $\Sigma$ is a critical point of $\mathcal{E}_0$ if and only 
 if $\Omega = \R^n \backslash K$ supports a solution to the over-determined 
 boundary value problem \eqref{eq:capacitor0}.
 \end {cor} 
 
  \begin {proof} Setting 
$$\Lambda^2 = \frac{(n-2)^2 \omega_n \operatorname{Cap}(K)}{|K|},$$
we use \eqref{cap-EL} to see that $K$ is a critical point of $\mathcal{E}_0$ 
if and only if 
$$\int_\Sigma \langle X, \eta \rangle \left [ \Lambda^2 - \left ( \frac{\partial U}
{\partial \eta} \right )^2 \right ] d\sigma = 0$$ 
for all possible variation fields $X$, which in turn implies 
\begin {equation} \label{cap-EL2} 
\frac{\partial U}{\partial \eta} = \Lambda = (n-2) \sqrt{\frac{\omega_n \operatorname{Cap}(K)}
{|K|}}\end {equation}
along $\Sigma$. We notice that if $K = \overline \B_\rho$ this constant 
reduces to $\Lambda = \frac{n-2}{\rho}$. 

  Conversely, let $\Omega$ admit a solution to \eqref{eq:capacitor0}. By the uniqueness of 
  solutions to \eqref{cap-pde}, this function must be the equilibrium 
  potential of $K$, so by \eqref{cap-EL} 
  $$ \left. \frac{d}{dt} \mathcal{E}_0(K_t) \right |_{t=0} = 0$$ for all 
  possible variation fields $X$.  
  \end {proof}

 Though we will not use it, we include the following derivation to satisfy the reader's 
  curiousity. 
\begin {lem} 
We have 
$$
-\left. \frac{d}{dt} \mathcal{E}_1 \right |_{t=0} 
= \frac{1}{|\Sigma|^{\frac{n-2}{n-1}}} \left [ \left ( \frac{n-2}{n-1}
\right ) \left ( \frac{\operatorname{Cap}(K)}{|\Sigma|} \right ) \int_\Sigma
v H_\Sigma d\sigma + \frac{1}{n(n-2)\omega_n} \int_\Sigma v \left ( 
\frac{\partial U}{\partial \eta} \right )^2 d\sigma \right ].$$ 
\end {lem} 

\begin {proof} 
We begin by differentiating the 
boundary condition $\left. U \right |_{\Sigma} = 1$ to see 
$\frac{\partial U}{\partial t} = - v \frac{\partial U}{\partial \eta} $ on $\Sigma$.
Thus
\begin {eqnarray*} 
\left. \frac{d}{dt} \mathcal{E}_1 \right |_{t=0} & = & 
\left. \frac{d}{dt} \right |_{t=0} \left ( \frac{1}{n(n-2)\omega_n}
|\Sigma|^{\frac{2-n}{n-1}} \int_{\Omega_t} \langle \nabla U_t, 
\nabla U_t \rangle dx \right ) \\ 
& = & \frac{1}{|\Sigma|^{\frac{n-2}{n-1}}} \left ( \frac{2-n}{n-1} \right ) |\Sigma|^{-1}
\frac{1}{n(n-2) \omega_n} \int_\Omega |\nabla U|^2 dx  \int_\Sigma vH_\Sigma 
d\sigma \\ 
&& + \frac{1}{n(n-2) \omega_n |\Sigma|^{\frac{n-2}{n-1}}} \left [ 2 
\int_\Omega \left \langle \nabla U, \nabla \frac{\partial U}{\partial t} 
\right \rangle dx + \int_\Sigma v |\nabla U|^2 d\sigma \right ]  \\ 
& = & - \left ( \frac{n-2}{n-1} \right ) \frac{\operatorname{Cap}(K)}
{|\Sigma|^{\frac{n-2}{n-1} + 1}} \int_\Sigma v H_\Sigma d\sigma  \\ 
&& + \frac{1}{n(n-2) \omega_n |\Sigma|^{\frac{n-2}{n-1}}}
\left [ \int_\Sigma v |\nabla U|^2 d\sigma - 2 \int_\Omega \frac{\partial U}
{\partial t} \Delta_0 U dx + 2 \int_\Sigma \frac{\partial U}{\partial t} 
\frac{\partial U}{\partial \eta} d\sigma \right ] \\
& = & - \frac{1}{|\Sigma|^{\frac{n-2}{n-1}}} \left [ \left ( \frac{n-2}{n-1} \right ) 
\frac{\operatorname{Cap}(K)}{|\Sigma|} \int_\Sigma vH_\Sigma d\sigma
+ \frac{1}{n(n-2)\omega_n} \int_\Sigma v \left ( \frac{\partial U}{\partial \eta} 
\right )^2 d\sigma \right ].
\end {eqnarray*}
Here we've used the fact that $H_\Sigma$ is the first variation of $|\Sigma|$ 
and that $|\nabla U| = - \frac{\partial U}{\partial \eta} $ on $\Sigma$. 
\end{proof} 

One can mimic the proof of Corollary \ref{over-determined-cor} to show 
that critical points of $\mathcal{E}_1$ are precisely those sets 
$K$ which admit a solution to the over-determined boundary value problem 
$$\Delta U = 0, \qquad \left. U \right |_{\partial K} = 1, \qquad \lim_{|x| 
\rightarrow \infty} U(x) = 0, \qquad \frac{\partial U}{\partial \eta} 
= \Lambda H,$$
where $H$ is the mean curvature of $\Sigma = \partial K$.

\end{document}